\documentclass[11pt]{article}

\usepackage{cite}
\usepackage{enumerate}
\usepackage{amsmath}
\usepackage{amsthm}
\usepackage{mathrsfs}
\usepackage{a4,latexsym,epsfig,amssymb,amsfonts}

\newtheorem{theorem}{Theorem}[section]

\newtheorem{lemma}[theorem]{Lemma}
\newtheorem{corollary}[theorem]{Corollary}

\newcommand{\F}{\mbox{$\cal F$}}

\title{Characterizations of minimal graphs with equal edge connectivity and spanning tree packing number
\thanks{ The paper was published with a different title ``Characterizations of strength extremal graphs'' in
Graphs and Combinatorics 30 (2014) 1453-1461.}}

\author{Xiaofeng Gu$^{1}$, 
Hong-Jian Lai$^{2,3}$, 
Ping Li$^{4}$, 
Senmei Yao$^5$\\
\footnotesize $^1$Department of Mathematics and Computer Science, University of Wisconsin-Superior,
Superior, WI 54880, USA\\
\footnotesize $^2$Department of Mathematics, West Virginia University, Morgantown, WV 26506, USA\\
\footnotesize $^3$College of Mathematics and System Sciences, Xinjiang University, Urumqi,
Xinjiang 830046, PRC\\
\footnotesize $^4$Department of Mathematics, Beijing Jiaotong University, Beijing 100044, PRC\\
\footnotesize $^5$Department of Mathematics, School of Arts and Sciences, Marian University, Fond du Lac, WI 54935, USA
}

\begin{document}
\date{}
\maketitle
\noindent

\begin{abstract}
With graphs considered as natural models for many network design problems,
edge connectivity $\kappa'(G)$ and maximum number of edge-disjoint spanning trees
$\tau(G)$ of a graph $G$ have been used as measures for reliability
and strength in communication networks modeled as
graph $G$ (see \cite{Cunn85, Matula87}, among others).
Mader \cite{Mader71} and Matula \cite{Matula72} introduced the maximum
subgraph edge connectivity 
$\overline{\kappa'}(G)=\max \{\kappa'(H): H \mbox{ is a subgraph of } G \}$.
Motivated by their applications in network design and by the established inequalities
\[
\overline{\kappa'}(G)\ge \kappa'(G) \ge \tau(G),
\]
we present the following in this paper:
\\
(i) For each integer $k>0$, a characterization for graphs $G$ with the property
that $\overline{\kappa'}(G) \le k$ but for any edge $e$
not in $G$, $\overline{\kappa'}(G+e)\ge k+1$.
\\
(ii) For any integer $n > 0$, a characterization for graphs
$G$ with $|V(G)| = n$ such that $\kappa'(G) = \tau(G)$ with $|E(G)|$ minimized.
\end{abstract}


\section{Introduction}

With graphs considered as natural models for many network design problems,
edge connectivity and maximum number of edge-disjoint spanning trees
of a graph have been used as measures for reliability
and strength in communication networks modeled as
a graph (see \cite{Cunn85, Matula87}, among others).

We consider finite graphs with possible multiple edges, and
follow notations of Bondy and Murty \cite{BoMu08}, 
unless otherwise defined. Thus for a graph $G$, $\omega(G)$ denotes
the number of components of $G$, and $\kappa'(G)$ denotes the edge connectivity of $G$.
For a connected graph $G$, $\tau(G)$ denotes the maximum number of edge-disjoint spanning 
trees in $G$. A survey on $\tau(G)$ can be found in \cite{Palmer01}.
By definition, $\tau(K_1)=\infty$. A graph $G$ is {\bf nontrivial} if $|E(G)|\neq\emptyset$.

For any graph $G$, we further define 
$\overline{\kappa'}(G)=\max \{\kappa'(H): H \mbox{ is a subgraph of } G \}$.
The invariant $\overline{\kappa'}(G)$, first introduced by Matula~\cite{Matula72}, has been
studied by Boesch and McHugh~\cite{BoMc85}, by Lai~\cite{Lai90}, by Matula~\cite{Matula72,Matula87},
by Mitchem~\cite{Mitchem77} and implicitly by Mader~\cite{Mader71}. In~\cite{Matula87}, Matula
gave a polynomial algorithm to determine $\overline{\kappa'}(G)$.

Throughout the paper, $k$ and $n$ denote positive integers, unless otherwise defined.

Mader \cite{Mader71} first introduced $k$-maximal graphs.
A graph $G$ is {\bf $k$-maximal} if $\overline{\kappa'}(G)\le k$ but for any edge
$e\not\in E(G)$, $\overline{\kappa'}(G+e)\ge k+1$.
The $k$-maximal graphs have been studied in 
\cite{BoMc85, Mader71,Matula72, Matula87, Mitchem77, Lai90}, among others.

Simple $k$-maximal graphs have been well studied. In~\cite{Mader71}, Mader proved that
the maximum number of edges in a simple $k$-maximal graph with $n$ vertices is
$(n-k)k+ {k\choose 2}$ and characterized all the extremal graphs.
In 1990, Lai~\cite{Lai90}
showed that the minimum number of edges in a simple $k$-maximal graph with $n$ vertices
is $(n-1)k-{k\choose 2}\lfloor \frac{n}{k+2}\rfloor$. In the same paper, Lai also 
characterized all extremal graphs and all simple $k$-maximal graphs.

In this paper, we mainly focus on multiple $k$-maximal graphs, and show that the number
of edges in a $k$-maximal graph with $n$ vertices is $k(n-1)$ and give a complete
characterization of all $k$-maximal graphs as well as show several equivalent graph families.

As it is known that for any 
connected graph $G$, $\kappa'(G) \ge \tau(G)$, it is natural to ask when the equality holds.
Motivated by this question, we characterize all graphs $G$ satisfying
$\kappa'(G)=\tau(G)$ with minimum number of possible edges for a fixed number of vertices.
We also investigate necessary and sufficient conditions for a graph to have
a spanning subgraph with this property or to be a spanning subgraph of another graph with
this property.

In Section 2, we display some preliminaries. In Section 3, we will characterize all
$k$-maximal graphs. The characterizations of minimal graphs with $\kappa'=\tau$ and
reinforcement problems will be discussed in Sections 4 and 5, respectively.

In this paper, an edge-cut always means a minimal edge-cut.

\section{Preliminaries}

Let $G$ be a nontrivial graph. The {\bf density} of $G$ is defined by
\begin{eqnarray}
\label{defdg}
d(G)=\frac{|E(G)|}{|V(G)|-\omega(G)}.
\end{eqnarray}
Hence, if $G$ is connected, then $d(G)=\frac{|E(G)|}{|V(G)|-1}$.
Following the terminology in \cite{CGHL92}, we define $\eta(G)$ and $\gamma(G)$
as follows:
\[
\eta(G)=\min\frac{|X|}{\omega(G-X)-\omega(G)} \mbox{ and }
\gamma(G)=\max \{d(H)\},
\]
where the minimum or maximum is taken over all edge subsets $X$ or subgraph $H$
whenever the denominator is non-zero. From the definitions of 
$d(G)$, $\eta(G)$ and $\gamma(G)$, we have, for any nontrivial graph $G$,
\begin{eqnarray}
\label{etdga}
\eta(G)\le d(G)\le \gamma(G).
\end{eqnarray}

As in \cite{CGHL92}, a graph $G$ satisfying $d(G)= \gamma(G)$ is said to be
{\bf uniformly dense}.
The following theorems are well known.

\begin{theorem} (Nash-Williams \cite{Nash61} and Tutte \cite{Tutte61})
\label{NaTu}
\\Let $G$ be a connected graph with $E(G)\neq \emptyset$, and let $k>0$ be
an integer. Then $\tau(G)\ge k$ if and only if for any $X\subseteq E(G),
|X|\ge k(\omega(G-X)-1)$.
\end{theorem}

Theorem~\ref{NaTu} indicates that for a connected graph $G$
\begin{eqnarray}
\label{taueta}
\tau(G)=\lfloor\eta(G)\rfloor.
\end{eqnarray}

\begin{theorem} (Catlin et al. \cite{CGHL92})
\label{cat92}
\\Let $G$ be a graph. The following statements are equivalent.
\\(i) $\eta(G)=d(G)$.
\\(ii) $d(G)=\gamma(G)$.
\\(iii) $\eta(G)=\gamma(G)$.
\end{theorem}

For a connected graph $G$ with $\tau(G) \ge k$, we define
$E_k(G) =\{e \in E(G): \tau(G-e) \ge k\}$.

\begin{lemma} (Lai et al. \cite{LaLL10b}, Li~\cite{Li12})
\label{Lailem}
\\ Let $G$ be a connected graph with $\tau(G)\ge k$. Then $E_k(G)=\emptyset$ if and only if $d(G)=k$.
\end{lemma}

\begin{lemma}(Haas \cite{Haas02}, Lai et al. \cite{LaLL10a} and Liu et al. \cite{LiLC09})
\label{Haaslem} 
\\Let $G$ be a graph, then the following statements are equivalent.
\\(i) $\gamma(G)\le k$.
\\(ii) There exist $k(|V(G)|-1)-|E(G)|$ edges whose addition to $G$ results in a graph
that can be decomposed into $k$ edge-disjoint spanning trees.
\end{lemma}


\section{Characterizations of $k$-maximal graphs}

In this section, we are to present a structural characterization of $k$-maximal graphs  
as well as several equivalent conditions, as shown in Theorem~\ref{charact}.

Let $F(n,k)$ be the maximum number of edges in a graph $G$ on $n$ vertices
with $\overline{\kappa'}(G)\le k$.
We define ${\cal F}(n,k)=\{G: |E(G)|=F(n,k), |V(G)|=n, \overline{\kappa'}(G)\le k\}$.

Let $G_1$ and $G_2$ be connected graphs such that $V(G_1)\cap V(G_2)=\emptyset$.
Let $K$ be a set of $k$ edges each of which has one vertex in $V(G_1)$ and the other 
vertex in $V(G_2)$. The {\bf $K$-edge-join} $G_1*_K G_2$ is defined to be the graph 
with vertex set $V(G_1)\cup V(G_2)$ and edge set $E(G_1)\cup E(G_2)\cup K$.
When the set $K$ is not emphasized, we use $G_1*_k G_2$ for $G_1*_K G_2$, and refer
to $G_1*_k G_2$ as a $k$-edge-join.

Let ${\cal G}_k$ be a family of graphs such that
for any $G_1, G_2\in {\cal G}_k\cup \{K_1\}, G_1 *_k G_2\in {\cal G}_k$.
Let $\overline{\tau}(G)=\max \{\tau(H): H \mbox{ is a subgraph of } G \}$.
The main theorem in this section is stated below.

\begin{theorem}
\label{charact}
Let $G$ be a graph on $n$ vertices. The following statements are equivalent.
\\(i) $G\in {\cal F}(n,k)$;
\\(ii) $G$ is $k$-maximal;
\\(iii) $\eta(G)=\overline{\kappa'}(G)=k$;
\\(iv) $\tau(G)=\overline{\kappa'}(G)=k$;
\\(v) $\tau(G)=\overline{\tau}(G)=\kappa'(G)=\overline{\kappa'}(G)=k$;
\\(vi) $G\in {\cal G}_k$.
\end{theorem}

In order to prove Theorem~\ref{charact}, we need some lemmas.

\begin{lemma}
\label{etylem}
Let $X$ be a $k$-edge cut of a graph $G$. If $H$ is a subgraph of $G$ with $\kappa'(H)>k$,
then $E(H)\cap X=\emptyset$.
\end{lemma}
\begin{proof}[{\bf Proof:}]
If $E(H)\cap X \neq\emptyset$, then $\kappa'(H)\le |E(H)\cap X|\le |X|=k< \kappa'(H)$,
a contradiction.
\end{proof}

\begin{lemma}
\label{kaplem}
If a graph $G$ is $k$-maximal, then $\kappa'(G)=\overline{\kappa'}(G)=k$.
\end{lemma}
\begin{proof}[{\bf Proof:}]
Since $G$ is $k$-maximal, $\kappa'(G)\le \overline{\kappa'}(G)\le k$.
It suffices to show that $\kappa'(G)=k$. We assume that $\kappa'(G)<k$
and prove it by contradiction. Let $X$ be an edge cut with $|X|<k$
and suppose that $G=G_1 *_X G_2$. Let $e\not\in E(G)$ be an edge with one end in $V(G_1)$
and the other end in $V(G_2)$. By the definition of $k$-maximal graphs, 
$\overline{\kappa'}(G+e)\ge k+1$. Thus $G+e$ has a subgraph $H$ with $\kappa'(H)\ge k+1$.
Then it must be the case that $e\in E(H)$, otherwise $H$ is a subgraph of $G$, contrary to
$\overline{\kappa'}(G)\le k$. 
Since $X\cup \{e\}$ is an edge cut of $G+e$ with $|X\cup \{e\}|\le k$
and $H$ is a subgraph of $G+e$ with $\kappa'(H)\ge k+1$, by Lemma~\ref{etylem}, 
$E(H)\cap (X\cup \{e\})=\emptyset$, contrary to $e\in E(H)$.
\end{proof}

\begin{lemma}
\label{induc}
If a graph $G$ is $k$-maximal, then 
$G=G_1*_k G_2$ where either $G_i=K_1$ or $G_i$ is $k$-maximal for $i=1,2$.
\end{lemma}
\begin{proof}[{\bf Proof:}]
By Lemma~\ref{kaplem}, $G$ has a $k$-edge cut $X$, and so $G=G_1 *_k G_2$.
For $i=1,2$, suppose that $G_i\neq K_1$, we want to prove that $G_i$ is $k$-maximal.
Since $G$ is $k$-maximal, $\overline{\kappa'}(G)\le k$, whence $\overline{\kappa'}(G_i)\le k$.
For any edge $e\not\in E(G_i)$, $\overline{\kappa'}(G+e)\ge k+1$.
Thus $G+e$ has a subgraph $H$ with $\kappa'(H)\ge k+1$. Since $\overline{\kappa'}(G)\le k$,
$H$ is not a subgraph of $G$, and so $e\in E(H)$. Since $X$ is a $k$-edge cut of
$G+e$, by Lemma~\ref{etylem}, $E(H)\cap X=\emptyset$. Hence $H$ is a subgraph 
of $G_i+e$ with $\kappa'(H)\ge k+1$, whence $\overline{\kappa'}(G_i)\ge k+1$.
Thus $G_i$ is $k$-maximal.
\end{proof}

\begin{lemma}
\label{equiva}
Let $G$ be a graph on $n$ vertices. 
Then $G\in {\cal F}(n,k)$ if and only if $G$ is $k$-maximal.
\end{lemma}
\begin{proof}[{\bf Proof:}]
By the definition of ${\cal F}(n,k)$, if $G\in {\cal F}(n,k)$, then $|E(G)|=F(n,k)$
and $\overline{\kappa'}(G)\le k$.
Then for any edge $e\not\in E(G)$, $|E(G+e)|=|E(G)|+1>F(n,k)$, and so 
$\overline{\kappa'}(G+e)\ge k+1$. By the definition of $k$-maximal graphs,
$G$ is $k$-maximal.

Now we assume that $G$ is $k$-maximal to prove that $G\in {\cal F}(n,k)$.
It suffices to show that any $k$-maximal graph $G$ has the property $\overline{\kappa'}(G)\le k$
with the maximum number of edges. We will prove that for any $k$-maximal graph $G$,
$|E(G)|=F(n,k)=k(n-1)$. We use induction on $n$.
When $n=2$, $G$ is $kK_2$, which is the graph with $2$ vertices and $k$ multiple edges,
and so $|E(G)|=k$. We assume that $|E(G)|=F(n,k)=k(n-1)$ holds for smaller values of $n>2$.
By Lemma~\ref{induc}, $G=G_1 *_k G_2$ where $G_i$ is $k$-maximal or $k_1$ for $i=1,2$.
Let $|V(G_i)|=n_i$. By inductive hypothesis, $|E(G_i)|=k(n_i-1)$. Thus 
$|E(G)|=k(n_1-1)+k(n_2-1)+k=k(n-1)$.
\end{proof}

\begin{corollary}
\label{catmax}
$F(n,k)=k(n-1)$.
\end{corollary}

\begin{lemma}
\label{overlem}
Suppose $\tau(G)=\overline{\tau}(G)=\kappa'(G)=\overline{\kappa'}(G)=k$.
Then $G=G_1*_k G_2$ where either $G_i=K_1$ or $G_i$ satisfies
$\tau(G_i)=\overline{\tau}(G_i)=\kappa'(G_i)=\overline{\kappa'}(G_i)=k$ for $i=1,2$.
\end{lemma}
\begin{proof}[{\bf Proof:}]
Since $\kappa'(G)=k$, there must be an edge-cut of size $k$. Hence
there exist graphs $G_1$ and $G_2$ such that $G=G_1*_k G_2$.
If $G_i\neq K_1$, we will prove 
$\tau(G_i)=\overline{\tau}(G_i)=\kappa'(G_i)=\overline{\kappa'}(G_i)=k$, for $i=1,2$.
First, by the definition of $\overline{\tau}$, 
$\tau(G_i)\le \overline{\tau}(G_i)\le \overline{\tau}(G)=k$ for $i=1,2$.
Since $G$ has $k$ disjoint spanning trees, we have $\tau(G_i)\ge k$ for $i=1,2$.
Thus $\tau(G_i)=\overline{\tau}(G_i)=k$ for $i=1,2$.
Now we prove $\kappa'(G_i)=\overline{\kappa'}(G_i)=k$ for $i=1,2$. 
Since $\overline{\kappa'}(G)=k$, $\kappa'(G_i)\le \overline{\kappa'}(G_i)\le k$. 
But $\kappa'(G_i)\ge \tau(G_i)=k$ for $i=1,2$. Hence we have
$\tau(G_i)=\overline{\tau}(G_i)=\kappa'(G_i)=\overline{\kappa'}(G_i)=k$ for $i=1,2$.
\end{proof}

\begin{lemma} 
\label{inverselem}
Let $G=G_1*_k G_2$ where $G_i=K_1$ or $G_i$ satisfies
$\tau(G_i)=\overline{\tau}(G_i)=\kappa'(G_i)=\overline{\kappa'}(G_i)=k$ for $i=1,2$.
Then $\tau(G)=\overline{\tau}(G)=\kappa'(G)=\overline{\kappa'}(G)=k$.
\end{lemma}
\begin{proof}[{\bf Proof:}]
Since $G=G_1*_k G_2$ and $\kappa'(G_1)=\kappa'(G_2)=k$, we have 
$\tau(G)\le \kappa'(G)=k$ and there exists an edge-cut $X=\{x_1,x_2,\cdots,x_k\}$
such that $G=G_1*_X G_2$. Let 
$T_{1,i}, T_{2,i},\cdots, T_{k,i}$ be edge-disjoint spanning trees of
$G_i$, for $i=1,2$.
Then $T_{1,1}+x_1+T_{1,2}, T_{2,1}+x_2+T_{2,2},\cdots,
T_{k,1}+x_k+T_{k,2}$ are $k$ edge-disjoint spanning trees of $G$. Thus
$\tau(G)=\kappa'(G)=k$. Now we need to prove that for any subgraph $H$ of $G$,
$\tau(H)\le k$ and $\kappa'(H)\le k$.
If $E(H)\cap X \neq\emptyset$, then $E(H)\cap X$ is an edge cut of $H$ and thus
$\tau(H)\le \kappa'(H)\le k$. If $E(H)\cap X =\emptyset$, then $H$ is a spanning
subgraph of either $G_1$ or $G_2$, whence $\tau(H)\le \kappa'(H)\le k$.
\end{proof}

Now we present the proof of Theorem~\ref{charact}.
\begin{proof}[{\bf Proof of Theorem~\ref{charact}:}]
By Lemma~\ref{equiva}, (i) and (ii) are equivalent. By (\ref{taueta}), (iii)$\Rightarrow$(iv).\\
(i)$\Rightarrow$(iii): By Corollary~\ref{catmax}, $|E(G)|=k(n-1)$. By the definition of $d(G)$,
$d(G)=k$. Since $\overline{\kappa'}(G)\le k$, for any subgraph $H$ of $G$, $\overline{\kappa'}(H)\le k$.
By Corollary~\ref{catmax}, $|E(H)|\le k(|V(H)|-1)$, whence $d(H)\le k$. By the definition of $\gamma(G)$,
we have $\gamma(G)\le k$. Thus $d(G)=\gamma(G)=k$. By Theorem~\ref{cat92}, $\eta(G)=k$.
Hence $k=\eta(G)=\tau(G)\le \overline{\kappa'}(G)\le k$, i.e., $\eta(G)=\overline{\kappa'}(G)=k$.\\
(iv)$\Rightarrow$(i): Since $\overline{\kappa'}(G)= k$, by Corollary~\ref{catmax}, $|E(G)|\le k(n-1)$. 
Since $\tau(G)=k$, $G$ has $k$ edge-disjoint spanning trees, and so $|E(G)|\ge k(n-1)$. 
Thus $|E(G)|=k(n-1)$, and so $G\in {\cal F}(n,k)$.\\
(iv)$\Leftrightarrow$(v): By definition, $\tau(G)\le \overline{\tau}(G)\le \overline{\kappa'}(G)$
and $\tau(G)\le\kappa'(G)\le \overline{\kappa'}(G)$. The equivalence between (iv) and (v) now follows
from these inequalities.\\
(v)$\Rightarrow$(vi): We argue by induction on $|V(G)|$. When $|V(G)|=2$, a graph $G$
with $\tau(G)=\overline{\tau}(G)=\kappa'(G)=\overline{\kappa'}(G)=k$ must be
$K_1*_k K_1$, and so by definition, $G\in {\cal G}_k$. 
We assume that (v)$\Rightarrow$(vi) holds for smaller values of $|V(G)|$.
By Lemma~\ref{overlem}, $G=G_1*_k G_2$ with
$\tau(G_i)=\overline{\tau}(G_i)=\kappa'(G_i)=\overline{\kappa'}(G_i)=k$
or $G_i=K_1$, for $i=1,2$. If $G_i\neq K_1$, then by the inductive hypothesis,
$G_i\in {\cal G}_k$. By definition, $G\in {\cal G}_k$.\\
(vi)$\Rightarrow$(v): We show it by induction on $|V(G)|$. 
When $|V(G)|=2$, by the definition of ${\cal G}_k$, $G=K_1*_k K_1$, and then
$\tau(G)=\overline{\tau}(G)=\kappa'(G)=\overline{\kappa'}(G)=k$.
We assume that it holds for smaller values of $|V(G)|$. By the definition of ${\cal G}_k$,
$G=G_1*_k K_1$ or $G=G_1*_k G_2$ where $G_1,G_2\in {\cal G}_k$. By inductive hypothesis, 
$\tau(G_i)=\overline{\tau}(G_i)=\kappa'(G_i)=\overline{\kappa'}(G_i)=k$ for $i=1,2$,
and by Lemma~\ref{inverselem}, $\tau(G)=\overline{\tau}(G)=\kappa'(G)=\overline{\kappa'}(G)=k$.
\end{proof}


\section{Characterizations of minimal graphs with $\kappa'=\tau$}

We define
\[
 \F_{k,n} = \{G: \kappa'(G) = \tau(G)=k, |V(G)|=n \mbox{ and
$|E(G)|$ is minimized} \}
\]
and $\F_{k}=\cup_{n > 1} \F_{k,n}$.

In this section, we will give characterizations of graphs in $\F_{k}$.
In addition, we use $\F_{k,n}$ to 
characterize graphs $G$ with $\kappa'(G) = \tau(G)$.

\begin{theorem}
\label{graphthm}
Let $G$ be a graph, then $G \in \F_{k}$ 
if and only if $G$ satisfies
\\(i) $G$ has an edge-cut of size $k$, and
\\(ii) $G$ is uniformly dense with density $k$.
\end{theorem}
\begin{proof}[{\bf Proof:}]
Suppose that $G \in \F_{k}$, then $\tau(G)=\kappa'(G)=k$.
Hence $G$ has an edge-cut of size $k$. Since $|E(G)|$ is minimized, we have
$E_k(G)=\emptyset$. By Lemma~\ref{Lailem}, $d(G)=k$. Since $\tau(G)=k$,
by Theorem~\ref{NaTu} and the definition of $\eta(G)$, we have $\eta(G)\ge k$.
By (\ref{etdga}), $\eta(G)\le d(G)=k$, whence $\eta(G)=d(G)=k$.
By Theorem~\ref{cat92}, $G$ is uniformly dense with density $k$.

On the other hand, suppose that $G$ satisfies (i) and (ii). By (ii) and 
Theorem~\ref{cat92}, $\eta(G)=d(G)=k$. By (\ref{taueta}), $\tau(G)=k$.
Then $\kappa'(G)\ge \tau(G)=k$. But $G$ has an edge-cut of size $k$, thus
$\kappa'(G)=\tau(G)=k$. Since $d(G)=k$, by Lemma~\ref{Lailem}, $E_k(G)=\emptyset$,
i.e. $|E(G)|$ is minimized. Thus $G \in \F_{k}$.
\end{proof}

\begin{theorem}
\label{kconthm}
A graph $G \in \F_{k}$ if and only if $G=G_1*_k G_2$ 
where either  $G_i=K_1$ or $G_i$ is uniformly dense with density $k$ for $i=1,2$.
\end{theorem}
\begin{proof}[{\bf Proof:}]
Suppose that $G \in \F_{k}$. By Theorem~\ref{graphthm},
$G$ has an edge-cut of size $k$, whence there exist graphs $G_1$ and $G_2$
such that $G=G_1*_k G_2$. Now we will prove that $G_i$ is uniformly dense with 
density $k$ if it is not isomorphic to $K_1$, for $i=1,2$. Since $\tau(G)=k$,
we have $\tau(G_i)\ge k$, and thus $d(G_i)\ge k$, for $i=1,2$. 
By (\ref{etdga}), (\ref{taueta}) and Theorem~\ref{cat92}, it suffices to prove
that $d(G_i)=k$ for $i=1,2$. If not, then either $d(G_1)>k$ or $d(G_2)>k$. By (\ref{defdg}),
$|E(G)|=|E(G_1)|+|E(G_2)|+k > k(|V(G_1)|-1)+k(|V(G_2)|-1)+k=k(|V(G)|-1)$,
and thus $d(G)=\frac{|E(G)|}{|V(G)|-1}>k$, contrary to the fact that $d(G)=k$.
Hence $d(G_i)=k$, and $k\le \tau(G_i)\le \eta(G_i)\le d(G_i)=k$.
By Theorem~\ref{cat92}, $G_i$ is uniformly dense with density $k$ for $i=1,2$. 
This proves the necessity.

To prove the sufficiency, first notice that $G$
must have an edge-cut of size $k$, by the definition of the $k$-edge-join.
In order to prove $G \in \F_{k}$,
by Theorem~\ref{graphthm}, it suffices to show that $G$ is uniformly dense with 
density $k$. Without loss of generality, 
we may assume that $G_i$ is not isomorphic to $K_1$ for $i=1,2$.
Then $\eta(G_i)=d(G_i)=k$ for $i=1,2$. By (\ref{taueta}), 
$\tau(G_i)=\lfloor\eta(G_i)\rfloor =k$.
Also we have $d(G_i)=\frac{|E(G_i)|}{|V(G_i)|-1}=k$ for $i=1,2$. Hence 
$E(G)=|E(G_1)|+|E(G_2)|+k=k(|V(G_1)|-1)+k(|V(G_2)|-1)+k=k(|V(G)|-1)$, whence
$d(G)=\frac{|E(G)|}{|V(G)|-1}=k$. Thus $k=\tau(G)\le \eta(G)\le d(G)=k$, i.e.,
$\eta(G)=d(G)=k$, and by Theorem~\ref{cat92}, $G$ is uniformly dense with density $k$. 
By Theorem~\ref{graphthm}, $G \in \F_{k}$.
\end{proof}

Theorem~\ref{kconthm} has the following corollary, presenting a recursive  structural 
characterization of graphs in $\F_{k}$.
\begin{corollary}
Let ${\cal K}(k)=\{G:\kappa'(G)>\eta(G)=d(G)=k\}$. 
Then a graph $G\in \F_{k}$ if and only if 
$G=((G_1 *_k G_2) *_k \cdots )*_k G_t$ for some integer $t\ge 2$ and $G_i\in {\cal K}(k)\cup \{K_1\}$
for $i=1,2,\cdots, t$.
\end{corollary}

Now we can characterize all the graphs $G$ with $\kappa'(G) = \tau(G) = k$.

\begin{theorem}
A graph $G$ with $n$ vertices satisfies $\kappa'(G) = \tau(G) = k$ if and only if 
$G$ has an edge-cut of size $k$ and a spanning subgraph in $\F_{k,n}$.
\end{theorem}
\begin{proof}[{\bf Proof:}]
First, suppose that $G$ satisfies $\kappa'(G) = \tau(G) = k$. Then $G$ must have an
edge-cut $C$ of size $k$ since $\kappa'(G)=k$. Hence,
$G=G_1*_C G_2$ where $\tau(G_i)\ge k$ or $G_i=K_1$ for $i=1,2$.
If $G_i=K_1$, then let $G'_i=K_1$. Otherwise,
$G_i$ must have $k$ edge-disjoint spanning trees $T_1,T_2,\cdots,T_k$, and let $G'_i$
be the graph with $V(G'_i)=V(G_i)$ and $E(G'_i)=\cup_{j=1}^k E(T_j)$.
Let $G'=G'_1*_C G'_2$. Then $G'$ is a spanning subgraph of $G$ with $\kappa'(G')=k$
and $k=\tau(G')\le \eta(G')\le d(G')=k$. By Theorem~\ref{graphthm}, 
$G'\in \F_{k}$. Since $|V(G')|=n$, $G'\in \F_{k,n}$,
completing the proof of necessity.

To prove the sufficiency, first notice that $\kappa'(G)\le k$,
since $G$ has an edge-cut of size $k$.
Graph $G$ has a spanning subgraph $G'\in \F_{k,n}$, so $\tau(G')=k$,
whence $\tau(G)\ge k$. Thus $k\le \tau(G)\le \kappa'(G)\le k$, and we have 
$\kappa'(G)=\tau(G)=k$.
\end{proof}

\section{Extensions and restrictions with respect to $\F_{k,n}$}

Let $G$ be a connected graph with $n$ vertices and $H\in \F_{k,n}$.
If $G$ is a spanning subgraph of $H$, then $H$ is an {\bf $\F_{k,n}$-extension}
of $G$. If $H$ is a spanning subgraph of $G$, then $H$ is an {\bf $\F_{k,n}$-restriction}
of $G$.

\begin{theorem}
Let $G$ be a connected graph with $n$ vertices. Then each of the following holds.
\\(i) $G$ has an $\F_{k,n}$-restriction if and only if
$G=G_1*_{k'}G_2$ for some $k'\ge k$ and graph $G_i$ with 
$\eta(G_i)\ge k$ or $G_i=K_1$, for $i=1,2$.
\\(ii) $G$ has an $\F_{k,n}$-extension if and only if 
$\kappa'(G)\le k$ and $\gamma(G)\le k$.
\end{theorem}
\begin{proof}[{\bf Proof:}]
(i) Suppose that $G$ has an $\F_{k,n}$-restriction $H$, by Theorem~\ref{kconthm}, $H=H_1*_k H_2$
where $\tau(H_i)=\eta(H_i)=d(H_i)=k$ or $H_i=K_1$ for $i=1,2$. Since
$H$ is a spanning subgraph of $G$, we have $G=G_1*_{k'} G_2$ for some $k'\ge k$ 
such that $H_i$ is a spanning subgraph of $G_i$ for $i=1,2$. 
If $H_i=K_1$, then $G_i=K_1$, otherwise,
$\eta(G_i)\ge \tau(G_i)\ge \tau(H_i)=k$ for $i=1,2$, by (\ref{taueta}).

To prove the sufficiency, it suffices to show that $G$ has a spanning subgraph $H\in \F_{k,n}$.
Since $G=G_1*_{k'}G_2$, there exists an edge-cut $X$ of size $k'$ such that $G=G_1*_{X}G_2$. 
Let $Y$ be a subset of size $k$ of $X$. For $i=1,2$, if $G_i=K_1$, then let $H_i=K_1$.
Otherwise, $\eta(G_i)\ge k$, and by (\ref{taueta}), $\tau(G_i)=\lfloor\eta(G_i)\rfloor\ge k$,
and then $G_i$ has $k$ edge-disjoint spanning trees $T_{1,i},T_{2,i},\cdots,T_{k,i}$. Let $H_i$ be the graph
with $V(H_i)=V(G_i)$ and $E(H_i)=\cup_{j=1}^k E(T_{j,i})$, for $i=1,2$. Let $H=H_1*_Y H_2$.
Then $H$ is a spanning subgraph of $G$ and $\kappa'(H)=\tau(H)=k$. Since $d(H)=k$, by Lemma~\ref{Lailem},
$H$ has the minimum number of edges with $\tau(H)=k$. Thus $H\in \F_{k,n}$.

(ii) If $G$ has an $\F_{k,n}$-extension $H$, then $G$ is a spanning subgraph of $H$ and 
$\kappa'(H)=\tau(H)=k$ with minimum number of edges. Then $\kappa'(G)\le k$. By Theorem~\ref{graphthm},
$d(H)=k$, i.e. $|E(H)|=k(|V(H)|-1)=k(|V(G)|-1)$. Thus $|E(H)|-|E(G)|=k(|V(G)|-1)-|E(G)|$, and
by Lemma~\ref{Haaslem}, $\gamma(G)\le k$.

To prove the sufficiency, it suffices to show that there is a graph $H\in \F_{k,n}$ with 
a spanning subgraph $G$. Let $\kappa'(G)=k'$, then $k'\le k$, and $G$ has an edge-cut $X$ of size $k'$.
Hence, $G=G_1*_X G_2$. For $i=1,2$, if $G_i=K_1$, then let $H_i=K_1$. Otherwise, 
since $\gamma(G)\le k$, by the definition of $\gamma(G)$, we have $\gamma(G_i)\le k$. 
By Lemma~\ref{Haaslem},
$G_i$ can be reinforcing to a graph $H_i$ which can be decomposed into $k$ edge-disjoint spanning trees.
Then $|E(H_i)|=k(|V(H_i)|-1)=k(|V(G_i)|-1)$, whence $d(H_i)=k$. Since $k=\tau(H_i)\le \eta(H_i)\le d(H_i)=k$,
we have $\eta(H_i)=d(H_i)=k$, and by Theorem~\ref{cat92}, $H_i$ is uniformly dense, for $i=1,2$.
Let $H=H_1*_Y H_2$ where $Y$ is an edge subset of size $k$ with $X\subseteq Y$. Then $G$ is a spanning subgraph of $H$.
By Theorem~\ref{kconthm}, $H\in \F_{k,n}$, and this completes the proof of the theorem.
\end{proof}

\section{Acknowledgment}
This work was part of X. Gu's PhD dissertation \cite{Gu13}, and was published with a different title
``Characterizations of strength extremal graphs'' in \cite{GLLY14}.


\begin{thebibliography}{99}


\bibitem{BoMc85}
E. T. Boesch and J. A. M. McHugh, An edge extremal result for subcohesion, J. Combinat. Theory Ser. B,
38 (1985), 1-7.

\bibitem{BoMu08}
J. A. Bondy and U. S. R. Murty, Graph Theory, Springer, New York, 2008.

\bibitem{CGHL92} 
P. A. Catlin, J. W. Grossman, A. M.
Hobbs and H.-J. Lai, Fractional arboricity, strength and principal
partitions in graphs and matroids, Discrete Appl. Math. 40 (1992)
285-302.

\bibitem{Cunn85} 
W. H. Cunningham, Optimal attack and reinforcement of a network,
Journal of the ACM, 32 (1985), 549-561.

\bibitem{Gu13} 
X. Gu, Connectivity and spanning trees of graphs, PhD Dissertation,
West Virginia University, 2013.

\bibitem{GLLY14}
X. Gu, H. -J. Lai, P. Li and S. Yao, Characterizations of strength extremal graphs,
Graphs and Combinatorics 30 (2014) 1453-1461.

\bibitem{Haas02}
R. Haas, Characterizations of arboricity of graphs, Ars Combinatoria, 63 (2002),
129-137.

\bibitem{Lai90}
H. -J. Lai, The size of strength-maximal graphs, Journal of Graph Theory, 14 (1990),
187-197.

\bibitem{LaLL10a} 
H. -J. Lai, P. Li, Y. Liang and J. Xu, Reinforcing a matroid to have
$k$ disjoint bases, Applied Mathematics, 1 (2010), 244-249.

\bibitem{LaLL10b} 
H. -J. Lai, P. Li and Y. Liang,
Characterization of removable elements with respect to having
$k$ disjoint bases in a matroid, Discrete Applied Math., 160 (2012), 2445-2451.

\bibitem{Li12}
P. Li, Cycles and bases of graphs and matroids, Ph.D. dissertation, West Virginia University, 2012.

\bibitem{LiLC09}
D. Liu, H. -J. Lai and Z.-H. Chen, Reinforcing the number of disjoint spanning trees,
Ars Combinatoria, 93 (2009), 113-127.

\bibitem{Mader71}
W. Mader, Minimale $n$-fach kantenzusammenh\"angende graphen, Math. Ann., 191 (1971), 21-28.

\bibitem{Matula72}
D. Matula, $K$-components, clusters, and slicings in graphs, SIAM J. Appl. Math., 22 (1972), 459-480.

\bibitem{Matula87}
D. Matula, Determining edge connectivity in $O(mn)$, Proceedings of
28th Symp. on Foundations of Computer Science, (1987), 249-251.

\bibitem{Mitchem77}
J. Mitchem, An extension of Brooks' theorem to $n$-degenerate graphs, Discrete Math., 17 (1977), 291-298

\bibitem{Nash61}
C. St. J. A. Nash-williams, Edge-disjoint spanning trees of finite graphs,
J. London Math. Soc., 36 (1961), 445-450.


\bibitem{Palmer01}
E. M. Palmer, On the spanning tree packing number of a graph, a survey,
Discrete Math., 230 (2001), 13-21.

\bibitem{Tutte61}
W. T. Tutte, On the problem of decomposing a graph into $n$ factors,
J. London Math. Soc., 36 (1961), 221-230.



\end{thebibliography}
\end{document}